\documentclass[reqno]{amsart}
\usepackage{amsmath, amssymb, amsthm, epsfig}
\usepackage{hyperref, latexsym}
\usepackage{url} 
\def\today{\ifcase\month\or
  January\or February\or March\or April\or May\or June\or
  July\or August\or September\or October\or November\or December\fi
  \space\number\day, \number\year}

 \newtheorem{theorem}{Theorem}
 \newtheorem{lemma}[theorem]{Lemma}

 \theoremstyle{definition}

 \theoremstyle{remark}
 
 \newcommand{\mc}{\mathcal}

 \newcommand{\R}{\mathbb{R}}

 \newcommand{\wg}{\widetilde{g}}
\newcommand{\wf}{\widetilde{f}}
 
 \newcommand{\wh}{\widetilde{h}}

\newcommand{\dz}{\text{\rm d}z}

 \newcommand{\dv}{\text{\rm d}v}

 \newcommand{\dx}{\text{\rm d}x}
 \newcommand{\dk}{\text{\rm d}k}

 \newcommand{\dom}{\text{\rm d}\omega}
 \newcommand{\dvs}{\text{\rm d} v_{*}}
\newcommand{\dsig}{\text{\rm d}\sigma_n^{\alpha}}

\newcommand{\dxi}{\text{\rm d}\xi_n^{b}}

\newcommand{\dnua}{\text{\rm d}\nu_{\alpha}}

 \newcommand{\dmu}{\text{\rm d}\mu(R)}
 \newcommand{\dnu}{\text{\rm d}\nu(x)}
 
 \newcommand{\dpi}{\text{\rm d}\pi}

 \newcommand{\fs}{f^{\star}}
\newcommand{\gs}{g^{\star}}
\newcommand{\hs}{h^{\star}}

\begin{document}

\title[Boltzmann collision operator]{Estimates for the Boltzmann collision operator via radial symmetry and Fourier transform}
\author[Alonso and Carneiro]{Ricardo J. Alonso* and Emanuel Carneiro**}
\thanks{*Research supported by NSF grant DMS-0636586}
\thanks{**Research supported by the CAPES/FULBRIGHT grant BEX 1710-04-4 and the Homer Lindsey Bruce Fellowship from the University of Texas.}
\date{\today}
\address{Department of Mathematics, University of Texas at Austin, Austin, TX 78712-1082.}
\email{ralonso@math.utexas.edu}
\address{Department of Mathematics, University of Texas at Austin, Austin, TX 78712-1082.}
\email{ecarneiro@math.utexas.edu}
\allowdisplaybreaks
\numberwithin{equation}{section}

\begin{abstract}
We extend the $L^p$-theory of the Boltzmann collision operator by using classical techniques based in the Carleman representation and Fourier analysis, allied to new ideas that exploit the radial symmetry of this operator. We are then able to greatly simplify existent technical proofs in this theory, extend the range, and obtain explicit sharp constants in some convolution-like inequalities for the gain part of the Boltzmann collision operator. 

\end{abstract} 

\maketitle
\section{Introduction}
\subsection{The Boltzmann equation}
Let us assume that we have a large space filled with particles that are considered as mass points.  Assume that these particles are interacting with a specific law and that the particles are not influenced by external forces.  A good model to represent such dynamical system is given by the equation
\begin{equation}\label{Intro1}
\frac{\partial{f}}{\partial{t}}+v\cdot\nabla_{x}{f}=Q(f,f)\;\;\mbox{in}\;\;(0,\infty)\times\mathbb{R}^{n}\times\mathbb{R}^{n}.
\end{equation}
The function $f(t,x,v)$, where $(t,x,v)\in(0,\infty)\times\mathbb{R}^{n}\times\mathbb{R}^{n}$, represents the phase space density of particles which at time $t$ and point $x$ move with velocity $v$.  The physical meaning implies that 
\begin{equation*}
f(t,x,v)\geq 0.
\end{equation*}
Equation (\ref{Intro1}) was derived by the first time by L. Boltzmann in 1872 in his studies of dilute gases. The term $Q(f,f)$ is known as the Boltzmann collision operator and its purpose it to model the interaction of the particles.  It is customary to split this operator in two, a positive and a negative part, which quantify the appearance and disappearance of particles in \textit{space-velocity} at a given time $t$.  Thus, for any suitable, measurable $f$ and $g$ we write
\begin{equation*}
Q(f,g):=Q^{+}(f,g)-Q^{-}(f,g)\,,
\end{equation*}
where
\begin{equation}\label{Intro2}
Q^{+}(f,g)(v):=\int_{\mathbb{R}^{n}}\int_{S^{n-1}}f(v')g(v'_{*})B(|u|,\hat{u}\cdot \omega)\,\dom\, \dvs\,,
\end{equation}
and 
\begin{equation}\label{Intro3}
Q^{-}(f,g)(v):=\int_{\mathbb{R}^{n}}\int_{S^{n-1}}f(v)g(v_{*})B(|u|,\hat{u}\cdot \omega)\, \dom \, \dvs.
\end{equation}
The pair of symbols $\{v',v'_{*}\}$ represents the final velocities of two particles after interacting with initial velocities $\{v,v_{*}\}$.  The relation between these is given by the formulas
\begin{equation*}\label{Intro4}
v'=V+\frac{|u|}{2}\omega \ \ \mbox{and} \ \ v'_{*}=V-\frac{|u|}{2}\omega\,,
\end{equation*}
where $V$ is the velocity of the center of mass of the particles, and $u$ is the relative velocity between them, i.e.
\begin{equation*}\label{Intro5}
V:=\frac{v+v_{*}}{2} \ \ \mbox{and} \ \ u:=v-v_{*}.
\end{equation*}
The symbol $\hat{u}$ represents the unitary vector in the direction of $u$ ($\hat{u}=u/|u|$) and $\dom$ is the surface measure on the sphere $S^{n-1}$.  The function $B(|u|,\hat{u}\cdot \omega)$ is known as the collision kernel and it is common to assume that this function can be factored in two: a magnitude function and an angular function, 
\begin{equation}\label{Intro6}
B(|u|,\hat{u}\cdot \omega)=\Phi(|u|)b(\hat{u}\cdot \omega).
\end{equation}
The most commom models found in the literature assume that $\Phi(|u|) = |u|^{\lambda}$, for example the Maxwellian molecules model ($\lambda=0$) and the hard spheres model ($\lambda=1$). Also, for the angular part, it is customary to assume that $b \geq 0$ and 
\begin{equation}\label{Intro7}
\int_{S^{n-1}}b(\hat{u}\cdot \omega)\,\dom < \infty.
\end{equation}
This condition, known as Grad's cut-off assumption, will be used throughout this paper.

\subsection{The Fourier transform approach}
The classical theory on Boltzmann equation establishes conservation of mass and energy for the solution. Therefore, the operator
\begin{equation*}\label{Intro8}
\frac{\partial}{\partial{t}}+v\cdot\nabla_{x}
\end{equation*}
admits a well-defined Fourier transform in velocity, for almost every $(t,x) \in (0,\infty) \times \R^n$, if applied to a solution of (\ref{Intro1}), namely
\begin{equation}\label{Intro9}
\frac{\partial{\widehat{f}}}{\partial{t}}+i\,\nabla_{x}\cdot\nabla_{k}\widehat{f}=\widehat{Q}(f,f)\;\;\mbox{in}\;\;(0,\infty)\times\mathbb{R}^{n}\times\mathbb{R}^{n}, 
\end{equation}
where $k$ is the variable in the Fourier space. This brings us to the problem of finding a reasonable representation for $\widehat{Q}(f,f)$; preferably in terms of $\widehat{f}$, since the left-hand side of the equation (\ref{Intro9}) depends only on $\widehat{f}$ (see \cite{D} for a complete discussion).

In the case of Maxwellian molecules such a representation was first figured by Bobylev in \cite{Bo1} and \cite{Bo2}. Denoting by $Q_0$ the collision operator in this case, he obtained
\begin{align}\label{Intro10}
\begin{split}
\widehat{Q_0}&(f,f)(k) =  \widehat{Q_0^+}(f,f)(k) -  \widehat{Q_0^-}(f,f)(k)\\
                    & = \int_{S^{n-1}}\widehat{f}(k^{+})\widehat{f}(k^{-})b(\hat{k}\cdot\omega)\, \dom- \widehat{f}(0) \widehat{f}(k)\left(\int_{S^{n-1}}b(\hat{u}\cdot \omega)\,\dom\right),
\end{split}
\end{align}
where $k^{+}$ and $k^{-}$ are given by
\begin{equation}\label{Intro11}
k^{+}=\frac{k+\left|k\right|\omega}{2} \ \ \ \textrm{and} \ \ \ k^{-}=\frac{k-\left|k\right|\omega}{2}.
\end{equation}

Our ultimate goal in this paper is to study the integrability properties of the positive part of the general Boltzmann collision operator defined in (\ref{Intro2}). In order to do this, we first study the Fourier transform of the gain term of the Maxwellian molecules operator
\begin{equation}\label{Intro11.1}
\widehat{Q_0^+}(f,f)(k) = \int_{S^{n-1}}\widehat{f}(k^{+})\widehat{f}(k^{-})b(\hat{k}\cdot\omega)\, \dom
\end{equation}
from a harmonic analysis point of view. Motivated by representation (\ref{Intro11.1}) we define the following operator, for continuous functions $g$ and $h$,
\begin{equation}\label{Intro12}
\mc{P}(g,h)(k)=\int_{S^{n-1}}g(k^{+})h(k^{-})b(\hat{k}\cdot\omega)\,\dom.
\end{equation}
The analysis of the bilinear operator $\mc{P}$ is the object of study in Section 2. The core result of this paper is Lemma \ref{Thm2.2}, a radial symmetrization inequality, that allows us to reduce the study of the operator $\mc{P}$ to radial variables. By doing so, we are naturally led to consider the following measure spaces. Let $b:[-1,1]\to \R^{+}$ be the angular part of the collision kernel, we will define the measure $\xi_n^{b}$ on $[0,1]$ by
\begin{equation}\label{Intro13.3}
\dxi(z) = b(2z-1)[z(1-z)]^{\tfrac{n-3}{2}}\,\dz.
\end{equation}
Most of the constants in our estimates will be given in terms of the following integral reminiscent of the classical beta function
\begin{equation}\label{Intro13.4}
\beta_b(x,y):= \int_0^1 z^x\,(1-z)^y\, \dxi(z),
\end{equation}
and, in this context, Grad's cut-off assumption (\ref{Intro7}) can be rewritten as
\begin{equation}\label{Intro13.1}
\int_{S^{n-1}}b(\hat{k}\cdot \omega)\,\dom = 2^{n-2}\left|S^{n-2}\right| \beta_b(0,0) < \infty\,.
\end{equation}
For $\alpha \in \R$, we will use the measure $\dnua(k) = |k|^{\alpha} \, \dk $ on $\R^n$, and further require
\begin{equation}\label{Intro13.2}
\beta_{b}\left(-\tfrac{n+\alpha}{2p},-\tfrac{n+\alpha}{2q}\right) < \infty \,
\end{equation}
to state our first result.
\begin{theorem}\label{thm1}
Let $1 \leq p,q,r \leq \infty$ with $1/p + 1/q = 1/r$, and $\alpha \in\R$. If the angular function $b:[-1,1]\to \R^{+}$ satisfies {\rm (\ref{Intro13.2})} the bilinear operator $\mc{P}$ extends to a bounded operator from $L^p(\R^n, \dnua) \times L^q(\R^n, \dnua)$ to $L^r(\R^n, \dnua)$ via the estimate
\begin{equation*}
\left\|\mathcal{P}(g,h)\right\|_{L^r(\R^n, \dnua)}\leq \, C \, \left\|g\right\|_{L^p(\R^n, \dnua)}\left\|h\right\|_{L^q(\R^n, \dnua)}.
\end{equation*}
The constant 
\begin{equation*}
 C = C(n,\alpha, p,q,b) = 2^{n-2}\left|S^{n-2}\right|\beta_{b}\left(-\tfrac{n+\alpha}{2p},-\tfrac{n+\alpha}{2q}\right)
\end{equation*}
is sharp.
\end{theorem}
Observe that if $\alpha > -n$ condition (\ref{Intro13.2}) implies (\ref{Intro13.1}), and vice versa if $\alpha < -n$. An interesting feature of Theorem 1 is that the sharp constant is found in terms of an integral condition for the kernel $b$ rather than classical pointwise assumptions (for example, that $b$ is bounded or vanishes near the endpoints). Similar integral conditions for other related inequalities (Povzner's lemmas) have been obtained in \cite{BGP}, \cite{GPV} and \cite{GPV2}.

\subsection{Young's inequality}

The $L^{p}$-theory of the Boltzmann collision operator started with the works \cite{C1} and \cite{C2} of Carleman in 1932 and 1957. Later, Arkeryd in \cite{A} extended the theory and worked $L^{\infty}$-estimates, but it was not until Gustafsson \cite{G} in 1988 that the convolution behavior of the Boltzmann collision operator was noticed. In his work, Gustafsson proves, by means of the Carleman representation \cite{C1} and the Riesz-Thorin interpolation theorem, estimates of the form\footnote{Inequalities (\ref{Intro14})-(\ref{Intro15}) are presented in an informal way. The precise statements involve weighted Lebesgue spaces and smooth conditions on the kernel $B$.}
\begin{equation}\label{Intro14}
\left\|\tilde{Q}^{+}(g,h)\right\|_{p}\leq C_{p}\left\|g\right\|_{1}\left\|h\right\|_{p},
\end{equation}
with $p\geq 1$ and for a truncated version $\tilde{Q}^{+}$ of the collision operator.  In the sequel, he uses O'Neil's interpolation result for convolutions \cite{ON} to conclude Young's inequality for this truncated operator:
\begin{equation}\label{Intro15}
\left\|\tilde{Q}^{+}(g,h)\right\|_{r}\leq C_{p,q}\left\|g\right\|_{p}\left\|h\right\|_{q},
\end{equation}
for all $p,q,r\geq 1$ such that $1/p+1/q=1+1/r$.  Since an intricate non-linear interpolation procedure is used in O'Neil's theorem, the constant $C_{p,q}$ is not explicit. More recently, Mouhot and Villani \cite{MV} studied extensions of these previous results to different weighted $L^{p}$ and Sobolev spaces.

We devote Section 3 of this paper to revisit and extend the $L^p$-theory of the Boltzmann collision operator, proving a more general version of Young's inequality previously obtained by Gustafsson (\cite[Lemma 2.2]{G}), Mouhot-Villani (\cite[Theorem 2.1]{MV}) and Gamba-Panferov-Villani (\cite[Lemma 4.1]{GPV2}). The novelty here relies mainly in two aspects: (a) we obtain Young's inequality for the full range $p,q,r$\,;\, (b) our explicit constant is once more given in terms of an integral condition in $b$, and therefore we do not have to assume that the kernel $b:[-1,1] \to \R^{+}$ is bounded or vanishes near the endpoints. Moreover, our proof is elementary and relies on the machinery developed in Section 2. We briefly describe this result below.

Let the weighted Lebesgue spaces $L^p_{\lambda}(\R^n)$ ($p\geq 1$, $\lambda \in \R$) be defined by the norm
\begin{equation}\label{Intro17}
\|f\|_{L^p_\lambda(\R^n)} = \left( \int_{\R^n} |f(k)|^p \,\left(1+ |k|^{p\lambda}\right) \dk\right)^{1/p}.
\end{equation}
Let $r' \geq 1$ be given. Recalling the integral operator $\beta_b$ defined in (\ref{Intro13.4}) and (\ref{Intro13.3}), we will make the following assumption on the angular kernel $b:[-1, 1] \to \R^{+}$
\begin{equation}\label{Intro18}
\beta_b\left(-\tfrac{n}{2r'}, -\tfrac{n}{2r'}\right) < \infty\,.
\end{equation}
\begin{theorem}\label{thm2}
Let $1 \leq p,\,q,\,r \leq \infty$ with $1/p + 1/q = 1 + 1/r$. Assume that $\Phi(|u|) = |u|^{\lambda}$ with $\lambda \geq 0$ and that the angular function $b:[-1,1] \to \R^{+}$ satisfies {\rm (\ref{Intro18})}. The bilinear operator $Q^{+}$ extends to a bounded operator from $L^p_{\lambda}(\R^n) \times L^q_{\lambda}(\R^n) \to L^r(\R^n)$ via the estimate
\begin{equation}\label{Intro19}
\left\| Q^{+}(g,h)\right\|_{L^r(\R^n)} \leq C \,\|g\|_{L^p_{\lambda}(\R^n)} \, \|h\|_{L^q_{\lambda}(\R^n)}.
\end{equation}
The constant $C$ may be taken as
\begin{equation*}
C = 2^{ \lambda+n-1}\, |S^{n-2}|\, \beta_b\left(-\tfrac{n}{2r'}, -\tfrac{n}{2r'}\right).
\end{equation*}
\end{theorem}


\section{Harmonic analysis approach to the Maxwellian molecules operator}

\subsection{Radial symmetrization techniques}
In this section we will work with an operator derived from the the Maxwellian molecules operator, in which $B(|u|,\hat{u}\cdot \omega)= b(\hat{u}\cdot \omega)$. Assume that the angular function $b:[-1,1] \to \R^{+}$ satisfies the Grad's cut-off assumption (\ref{Intro7}). For continuous functions $g$ and $h$ we define the bilinear operator, for $k \neq 0$,
\begin{equation}\label{S2.3}
\mc{P}(g,h)(k)=\int_{S^{n-1}}g(k^{+})h(k^{-})b(\hat{k}\cdot\omega)\,\dom,
\end{equation}
where $k^{+}$ and $k^{-}$ are given by
\begin{equation}\label{S2.4}
k^{+}=\frac{k+\left|k\right|\omega}{2} \ \ \ \textrm{and} \ \ \ k^{-}=\frac{k-\left|k\right|\omega}{2}.
\end{equation}
Recall here that we are denoting $\hat{k}$ as the unitary vector in the direction of $k$ (i.e. $\hat{k} = k/|k|$). From (\ref{S2.4}) we can easily infer that 
\begin{equation}\label{S2.5}
k=k^{+}+k^{-} \ \ \ \textrm{and} \ \ \ \left|k\right|^{2}=\left|k^{+}\right|^{2}+\left|k^{-}\right|^{2}. 
\end{equation}

The purpose of this section is to study the operator $\mc{P}$ defined in (\ref{S2.3}), which can be seen as a special kind of convolution in the sphere. Motivated by the Riesz rearrangement inequality for the classical convolution, one might expect that the radial symmetry should also play a role here, namely, that we should be able to relate $\mc{P}(g,h)$ with $\mc{P}(g^{\star},h^{\star})$ where $g^{\star}$ and $h^{\star}$ are suitable radial symmetrizations of $g$ and $h$. This is indeed the case, and in order to clarify this behavior, we start with the following Carleman type lemma.

\begin{lemma}\label{lem2.1}
Let $f$, $g$ and $h$ be in $C_0(\R^n)$ and $b$ in $C([-1,1])$. Then
\begin{align}\label{S2.6}
\begin{split}
&\int_{\mathbb{R}^{n}} f(k)\mathcal{P}(g,h)(k)\,\dk = \\&   2^{n-1}\int_{\mathbb{R}^{n}}\frac{g(x)}{\left|x\right|}\int_{\left\{x\cdot z=0\right\}}\frac{f(x+z)}{\left|x+z\right|^{n-2}}\;h(z)\;b\left(\tfrac{2|x|^{2}}{|x+z|^{2}}-1\right)\dpi_z\,\dx\, ,
\end{split}
\end{align}
where $\dpi_z$ denotes the $(n-1)$-dimensional Lebesgue measure in the hyperplane $\{x\cdot z=0\}$.
\end{lemma}
\begin{proof}
We follow closely the ideas of Gamba, Panferov and Villani for Carleman's representation in \cite[Lemma 16]{GPV}. For a continuous function $\phi$ we have
\begin{equation}\label{S2.7}
\int_{S^{n-1}}\phi(\omega)\, \dom =\int_{\mathbb{R}^{n}}\phi(z)\, \delta\left(\tfrac{\left|z\right|^{2}-1}{2}\right) \dz
\end{equation}
where $\delta(z)$ is the one-dimensional Dirac measure. From (\ref{S2.7}) we obtain 
\begin{equation*}\label{S2.8}
\int_{\mathbb{R}^{n}} f(k)\mathcal{P}(g,h)(k)\,\dk =\int_{\mathbb{R}^{n}}\int_{\mathbb{R}^{n}}f(k)g(k^{+})h(k^{-})\,\delta\left(\tfrac{\left|z\right|^{2}-1}{2}\right)b(\hat{k}\cdot \hat{z})\,\dz\,\dk\, ,
\end{equation*}
with $k^{\pm}=\frac{k\pm\left|k\right|z}{2}$. We further set $x = k^{+}$. For every $k\neq 0$ fixed, this defines a linear map $z \mapsto x$ with determinant $\left(\frac{\left|k\right|}{2}\right)^{n}$. Using this change of variables we conclude that the previous integral is equal to 
\begin{align}\label{S2.11}
\begin{split}
& \int_{\mathbb{R}^{n}}\int_{\mathbb{R}^{n}}f(k)g(x)h(k-x)\,\delta\left(\tfrac{2(\left|x\right|^{2}-x\cdot k)}{\left|k\right|^{2}}\right)b\left(\hat{k}\cdot \hat{(2x - k)}\right)\left(\tfrac{2}{\left|k\right|}\right)^{n}\dk\,\dx.\\
& = 2^{n-1}\int_{\mathbb{R}^{n}}\int_{\mathbb{R}^{n}}\frac{f(k)}{\left|k\right|^{n-2}}\;g(x)h(k-x)\,\delta\left(\left|x\right|^{2}-x\cdot k\right)b\left(\hat{k}\cdot \hat{(2x - k)}\right)\dk\, \dx.
\end{split}
\end{align}
We now use a second change of variables, $z = k - x$, in (\ref{S2.11}) to obtain
\begin{equation*}\label{S2.12}
2^{n-1}\int_{\mathbb{R}^{n}}\int_{\mathbb{R}^{n}}\frac{f(x+z)}{\left|x+z\right|^{n-2}}\;g(x)h(z)\delta\left(x\cdot z\right)b\left(\hat{(x+z)} \cdot \hat{(x-z)}\right)\dz\,\dx.
\end{equation*}
To conclude, observe that, for $x\neq 0$ and any test function $\phi$,
\begin{equation*}\label{S2.13}
\int_{\mathbb{R}^{n}}\phi(z)\delta(x\cdot z)\,\dz=\left|x\right|^{-1}\int_{\{x\cdot z=0\}}\phi(z)\, \dpi_z\,.
\end{equation*}
\end{proof}
We are now ready to define the radial symmetrizations that will be used in this section. Let $G = SO(n)$ be the group of rotations of $\R^n$, in which we will use the variable $R$ to designate a generic rotation. We assume that the Haar measure $\textrm{d} \mu$ of this compact topological group is normalized so that 
\begin{equation}\label{S2.14}
\int_{G} \dmu = 1.
\end{equation}
Let $f \in L^p(\R^n)$, $p\geq 1$. We define the radial symmetrization $\fs_p$ by
\begin{equation}\label{S2.15}
\fs_p(x) = \left(\int_{G} |f(Rx)|^p \,\dmu\right)^{\tfrac{1}{p}}\,, \ \ \textrm{if} \ \ 1\leq p< \infty.
\end{equation}
and
\begin{equation}\label{S2.15.1}
\fs_{\infty}(x) = \textrm{ess sup}_{|y| =|x|} |f(y)|\,,
\end{equation}
where the essential sup in (\ref{S2.15.1}) is taken over the sphere of radius $|x|$ with respect to the surface measure over this sphere. The new function $\fs_p$ defined in (\ref{S2.15}) can be seen as an $L^p$-average of $f$ over all the rotations $R \in G$ and it satisfies the following properties:
\begin{itemize}
 \item[(i)] $\fs_p$ is radial.
\item[(ii)] If $f$ is continuous (or compactly supported) then $\fs_p$ is also continuous (or compactly supported).
\item[(iii)] If $g$ is a radial function then $(fg)^{\star}_p(x) = \fs_p(x)g(x)$.
\item[(iv)] Let $\textrm{d} \nu$ be a rotationally invariant measure on $\R^n$. Then
$$\int_{\R^n} |f(x)|^p \, \dnu = \int_{\R^n} |\fs_p(x)|^p \, \dnu.$$
In particular,
\begin{equation}\label{S2.15.2}
\|f\|_{L^p(\R^n)} = \|\fs_p\|_{L^p(\R^n)}.
\end{equation}
\end{itemize}
\begin{lemma}\label{Thm2.2}
Let $f$, $g$, $h$ be in $C_0(\R^n)$, b in $C([-1,1])$, and $1/p \,+ \,1/q \,+\, 1/r = 1$, with $1\leq p,q,r \leq \infty$. Then
\begin{equation}\label{S2.16}
\left|\int_{\mathbb{R}^{n}}f(k)\,\mathcal{P}(g,h)(k)\,\dk \right|\leq \int_{\mathbb{R}^{n}}\fs_p(k)\,\mathcal{P}(\gs_q,\hs_r)(k)\, \dk.
\end{equation}
\end{lemma}
\begin{proof}
We use here representation (\ref{S2.6}). If $R$ is a rotation in $\R^n$, by a change of variables we obtain
\begin{align}\label{S2.17}
\begin{split}
&\left|\int_{\mathbb{R}^{n}} f(k)\mathcal{P}(g,h)(k)\,\dk \right| \\ 
& = 2^{n-1}\left|\int_{\mathbb{R}^{n}}\frac{g(x)}{\left|x\right|}\int_{\left\{x\cdot z=0\right\}}\frac{f(x+z)}{\left|x+z\right|^{n-2}}\;h(z)\;b\left(\tfrac{2|x|^{2}}{|x+z|^{2}}-1\right)\dpi_z\,\dx \right| \\
& = 2^{n-1}\left|\int_{\mathbb{R}^{n}}\frac{g(Rx)}{\left|x\right|}\int_{\left\{x\cdot z=0\right\}}\frac{f(Rx + Rz)}{\left|x+z\right|^{n-2}}\;h(Rz)\;b\left(\tfrac{2|x|^{2}}{|x+z|^{2}}-1\right)\dpi_z\,\dx\right| \\
& \leq 2^{n-1}\int_{\mathbb{R}^{n}}\frac{|g(Rx)|}{\left|x\right|}\int_{\left\{x\cdot z=0\right\}}\frac{|f(Rx + Rz)|}{\left|x+z\right|^{n-2}}\;|h(Rz)|\;b\left(\tfrac{2|x|^{2}}{|x+z|^{2}}-1\right)\dpi_z\,\dx.
\end{split}
\end{align}
Observe that the left hand side of (\ref{S2.17}) does not depend on the rotation $R$. Therefore, when we integrate over the group $G = SO(n)$ using (\ref{S2.14}) we find
\begin{align}\label{S2.18}
\begin{split}
&\left|\int_{\mathbb{R}^{n}} f(k)\mathcal{P}(g,h)(k)\,\dk\right| = \int_G \left|\int_{\mathbb{R}^{n}} f(k)\mathcal{P}(g,h)(k)\,\dk\right| \,\dmu\\   
& \leq 2^{n-1}\int_G \int_{\mathbb{R}^{n}}\frac{|g(Rx)|}{\left|x\right|}\int_{\left\{x\cdot z=0\right\}}\frac{|f(Rx + Rz)|}{\left|x+z\right|^{n-2}}\;|h(Rz)|\,b\left(\tfrac{2|x|^{2}}{|x+z|^{2}}-1\right)\dpi_z\,\dx \,\dmu.
\end{split}
\end{align}
By Fubini's theorem and H\"{o}lder's inequality we see that the right hand side of (\ref{S2.18}) is 
\begin{align*}\label{S2.19}
\begin{split}
= &\, \,2^{n-1}\int_{\mathbb{R}^{n}}\int_{\left\{x\cdot z=0\right\}}  \int_G |g(Rx)|\,|f(Rx + Rz)|\,|h(Rz)|\,\dmu\,\frac{b\left(\tfrac{2|x|^{2}}{|x+z|^{2}}-1\right)}{\left|x\right|\,\left|x+z\right|^{n-2}}\,\dpi_z\,\dx\\
\leq & \,\,2^{n-1}\int_{\mathbb{R}^{n}}\int_{\left\{x\cdot z=0\right\}} \left(\int_G |g(Rx)|^q\, \dmu\right)^{\tfrac{1}{q}}  \left( \int_G |f(Rx + Rz)|^p\, \dmu \right)^{\tfrac{1}{p}} \\
&   \ \ \ \ \ \  \ \ \ \ \ \ \ \  \ \ \ \  \  \left( \int_G |h(Rz)|^r \, \dmu \right)^{\tfrac{1}{r}}\ \ \frac{b\left(\tfrac{2|x|^{2}}{|x+z|^{2}}-1\right)}{\left|x\right|\,\left|x+z\right|^{n-2}}\,\dpi_z\,\dx\\
= &  \, \,2^{n-1}\int_{\mathbb{R}^{n}}\int_{\left\{x\cdot z=0\right\}} \frac{\gs_q(x)}{\left|x\right|}\frac{\fs_p(x + z)}{\left|x+z\right|^{n-2}}\,\,  \hs_r(z)\, b\left(\frac{2|x|^{2}}{|x+z|^{2}}-1\right) \dpi_z\,\dx\\ 
= & \int_{\mathbb{R}^{n}}\fs_p(k)\,\mathcal{P}(\gs_q,\hs_r)(k)\, \dk\,,
\end{split}
\end{align*}
and this concludes the proof.
\end{proof}

Lemma \ref{Thm2.2} shows that, in order to obtain $L^p$-estimates for the operator $\mc{P}$, it suffices to consider its action on radial functions. We explain briefly how to reduce this problem to a one-dimensional analogue, and as we move on, we introduce some additional notation.

Let $f:\R^n \to \R^n$ be a radial function. We define the function $\wf:\R^{+} \to \R$ by
\begin{equation}\label{S2.20}
f(k) = \wf(|k|^2).
\end{equation}
Observe that for any $p\geq 1$ and $\alpha \in \R$ we have
\begin{align}\label{S2.21}
\begin{split}
\int_{\mathbb{R}^{n}}f(k)^{p}\,|k|^{\alpha}\,\dk & =\int_{S^{n-1}}\int^{\infty}_{0}\wf(\left|k\right|^{2})^{p}\,\left|k\right|^{n+\alpha-1}\textrm{d}\!\left|k\right|\dom \\
& =\tfrac{\left|S^{n-1}\right|}{2}\int^{\infty}_{0}\wf(x)^{p}\,\dsig(x),
\end{split}
\end{align}
where 
\begin{equation}\label{S2.22}
\dsig(x)=x^{(n+\alpha-2)/2}\,\dx.
\end{equation}
Hence, if we define the measure $\nu_{\alpha}$ on $\R^n$ by
\begin{equation}\label{S2.23}
\dnua(k) = |k|^{\alpha} \, \dk,
\end{equation}
equation (\ref{S2.21}) translates to 
\begin{equation}\label{S2.24}
\|f\|_{L^p(\R^n, \,\dnua)} = \Bigl(\tfrac{\left|S^{n-1}\right|}{2}\Bigr)^{\tfrac{1}{p}}\,\|\wf\ \|_{L^{p}(\R^{+},\,\dsig)}.
\end{equation}
From definitions (\ref{S2.3}) and (\ref{S2.20}) we observe that for radially symmetric functions $g$ and $h$ we have
\begin{align}\label{S2.25}
\begin{split}
\mathcal{P}(g&,h)(k)=\int_{S^{n-1}}\wg(\left|k^{+}\right|^{2})\,\wh(\left|k^{-}\right|^{2})\,b(\hat{k}\cdot\omega)\,\dom\\
&=\int_{S^{n-1}}\wg\left(\left|k\right|^{2}\frac{1+\hat{k}\cdot\omega}{2}\right)\,\wh\left(\left|k\right|^{2}\frac{1-\hat{k}\cdot\omega}{2}\right)b(\hat{k}\cdot\omega)\,\dom\\
& = \left|S^{n-2}\right|\int_{-1}^{1}\wg\left(\left|k\right|^{2}\frac{1+ s}{2}\right)\,\wh\left(\left|k\right|^{2}\frac{1-s}{2}\right)b(s)\,(1-s^2)^{\tfrac{n-3}{2}}\, \textrm{d} s\\
& = 2^{n-2}\left|S^{n-2}\right|\int_{0}^{1}\wg\bigl(\left|k\right|^{2}z\bigr)\,\wh\bigl(\left|k\right|^{2} (1-z) \bigr)\,b(2z-1)\,[z(1-z)]^{\tfrac{n-3}{2}}\, \textrm{d} z.
\end{split}
\end{align}
By defining the new measure $\xi_n^{b}$ on $[0,1]$,
\begin{equation}\label{S2.26}
\dxi(z) = b(2z-1)\,[z(1-z)]^{\tfrac{n-3}{2}}\, \textrm{d} z\, ,
\end{equation}
and using (\ref{S2.20}), we can rewrite equation (\ref{S2.25}) as 
\begin{equation}\label{S2.27}
\widetilde{\mc{P}(g,h)}(x) = 2^{n-2}\left|S^{n-2}\right|\int_{0}^{1}\wg(xz)\,\wh\bigl(x (1-z)\bigr)\,\dxi(z).
\end{equation}
The purpose of the next subsection is to study the new integral operator defined in (\ref{S2.27}).
\subsection{The bilinear operator $\mc{B}(g,h)$}
Motivated by (\ref{S2.27}), for functions $g:\R^{+} \to \R$ and $h:\R^{+} \to \R$, we define $\mc{B}(g,h):\R^{+} \to \R$ by
\begin{equation}\label{S2.28}
\mc{B}(g,h)(x) = \int_{0}^{1}g(xz)\,h\bigl(x (1-z)\bigr)\,\dxi(z).
\end{equation}
In what follows we will use the function $\beta_b(x,y)$  already defined in the Introduction of this paper
\begin{equation}\label{S2.29}
\beta_{b}(x,y):=\int^{1}_{0}z^{x}\,(1-z)^{y}\,\dxi(z).
\end{equation}
The main result of this subsection is described below.
\begin{lemma}\label{Thm2.3}
For $g \in L^p(\R^{+}, \dsig)$ and $h \in L^q(\R^{+}, \dsig)$, we have

\begin{equation}\label{S2.30}
\left\|\mathcal{B}(g,h)\right\|_{L^r(\R^{+},\, \dsig)}\leq \beta_{b}\left(-\tfrac{n+\alpha}{2p},-\tfrac{n+\alpha}{2q}\right)\left\|g\right\|_{L^p(\R^{+},\, \dsig)}\left\|h\right\|_{L^q(\R^{+},\, \dsig)}.
\end{equation}
\\
where $1/p + 1/q = 1/r$, with $1\leq p,q,r\leq \infty$. The constant
\begin{equation}\label{S2.30.2}
 C(n,\alpha,p,q,b)= \beta_{b}\left(-\tfrac{n+\alpha}{2p},-\tfrac{n+\alpha}{2q}\right)
\end{equation}
is sharp.
\end{lemma}

\begin{proof}
Using Minkowski's inequality we obtain
\begin{align}\label{S2.31}
\begin{split}
\left\|\mathcal{B}(g,h)\right\|&_{L^r(\R^{+},\, \dsig)} \leq \left(\int^{\infty}_{0}\left(\int^{1}_{0}\left|g(xz)\right|\left|h\bigl(x(1-z)\bigr)\right|\dxi(z)\right)^{r}\dsig(x)\right)^{\tfrac{1}{r}}\\
& \leq \int^{1}_{0}\left(\int^{\infty}_{0}\left|g(xz)\right|^{r}\left|h\bigl(x(1-z)\bigr)\right|^{r}\dsig(x)\right)^{\tfrac{1}{r}}\dxi(z).
\end{split}
\end{align}
Next, we use H\"{o}lder's inequality with exponents $p/r$ and $q/r$ in the inner integral
\begin{align}\label{S2.32}
\begin{split}
& \left(\int^{\infty}_{0} \left|g(xz)\right|^{r}\left|h\bigl(x(1-z)\bigr)\right|^{r}\dsig(x)\right)^{\tfrac{1}{r}}\\
& \ \ \ \ \ \ \ \  \ \  \ \ \ \leq  \left(\int^{\infty}_{0}\left|g(xz)\right|^{p}\dsig(x)\right)^{\tfrac{1}{p}}\left(\int^{\infty}_{0}\left|h\bigl(x(1-z)\bigr)\right|^{q}\dsig(x)\right)^{\tfrac{1}{q}}\\
&\ \ \ \ \ \ \ \  \ \  \ \ \ = z^{-\tfrac{n+\alpha}{2p}}\, (1-z)^{-\tfrac{n+\alpha}{2q}}\, \left\|g\right\|_{L^p(\R^{+},\, \dsig)}\left\|h\right\|_{L^q(\R^{+},\, \dsig)}.\\
\end{split}
\end{align}
The boundedness of the operator $\mc{B}$ proposed in (\ref{S2.30}) follows easily from (\ref{S2.31}) and (\ref{S2.32}).

To prove that the constant $C(n,\alpha,p,q,b)$ defined in (\ref{S2.30.2}) is indeed sharp, we exhibit a pair of sequences $\{g_{\epsilon}\}$ and $\{h_{\epsilon}\}$ with $\epsilon \to 0$ satisfying 
\begin{equation}\label{S2.33}
 \left\|g_{\epsilon}\right\|_{L^p(\R^{+},\, \dsig)} = \left\|h_{\epsilon}\right\|_{L^q(\R^{+},\, \dsig)} = 1\,,
\end{equation}
for any $\epsilon >0$, and 
\begin{equation}\label{S2.34}
 \displaystyle\lim_{\epsilon \to 0} \left\|\mathcal{B}(g_{\epsilon},h_{\epsilon})\right\|_{L^r(\R^{+},\, \dsig)} = C(n,\alpha,p,q,b).
\end{equation}
Define the sequences by
\begin{equation*}\label{S2.35}
g_{\epsilon}(x)=\left\{
\begin{array}{cl}
\epsilon^{1/p}\;x^{-(n+\alpha-2\epsilon)/2p}&\mbox{for}\;\;0< x< 1\, ,\\
0&\mbox{otherwise},
\end{array}
\right.
\end{equation*}
and
\begin{equation*}\label{S2.36}
h_{\epsilon}(x)=\left\{
\begin{array}{cl}
\epsilon^{1/q}\;x^{-(n+\alpha- 2\epsilon)/2q}&\mbox{for}\;\;0< x< 1,\\
0&\mbox{otherwise}.
\end{array}
\right.
\end{equation*}
A direct computation shows (\ref{S2.33}). In order to prove (\ref{S2.34}), we estimate $\mathcal{B}(g_{\epsilon},h_{\epsilon})(x)$ in three different intervals, namely:\\
\\
\underline{For $0<x\leq 1$:} In this interval,
\begin{align}\label{S2.37}
\begin{split}
\mathcal{B}(g_{\epsilon},h_{\epsilon})(x)&=\epsilon^{1/r}\;x^{-(n+\alpha-2\epsilon)/2r}\int^{1}_{0}z^{-(n+\alpha-2\epsilon)/2p}\;(1-z)^{-(n+\alpha-2\epsilon)/2q}\, \dxi(z)\\
&=\epsilon^{1/r}\;x^{-(n+\alpha-2\epsilon)/2r}\beta_{b}\left(-\tfrac{n+\alpha-2\epsilon}{2p},-\tfrac{n+\alpha-2\epsilon}{2q}\right).
\end{split}
\end{align}
\underline{For $1<x\leq 2$:} In this interval we use the same estimate as before 
\begin{align}\label{S2.38}
\begin{split}
\mathcal{B}(g_{\epsilon},h_{\epsilon})(x)&=\epsilon^{1/r}\;x^{-(n+\alpha-2\epsilon)/2r}\int^{1/x}_{1-1/x}z^{-(n+\alpha-2\epsilon)/2p}\;(1-z)^{-(n+\alpha-2\epsilon)/2q}\, \dxi(z)\\
&\leq \epsilon^{1/r}\;x^{-(n+\alpha-2\epsilon)/2r}\beta_{b}\left(-\tfrac{n+\alpha-2\epsilon}{2p},-\tfrac{n+\alpha-2\epsilon}{2q}\right).
\end{split}
\end{align}
\underline{For $x>2$:} Here we have
\begin{equation*}\label{S2.39}
\mathcal{B}(g_{\epsilon},h_{\epsilon})(x) =0.
\end{equation*}
Therefore,
\begin{align*}\label{S2.40}
\begin{split}
\left\|\mathcal{B}(g_{\epsilon},h_{\epsilon})\right\|^{r}_{L^r(\R^{+}, \dsig)}& =\int^{1}_{0}\mathcal{B}(g_{\epsilon},h_{\epsilon})(x)^{r} \dsig(x)+\int^{2}_{1}\mathcal{B}(g_{\epsilon},h_{\epsilon})(x)^{r}\dsig(x)\\
\\
& :=\mbox{ (I)}_{\epsilon}+\mbox{ (II)}_{\epsilon}.
\end{split}
\end{align*}
From (\ref{S2.37}) and (\ref{S2.38}) we conclude that, as $\epsilon \to 0$,
\begin{equation*}\label{S2.41}
\mbox{(I)}_{\epsilon}\to \beta_{b}\left(-\tfrac{n+\alpha}{2p},-\tfrac{n+\alpha}{2q}\right)^{r} \ \ \ \ \textrm{and}\ \ \ \ \ \mbox{(II)}_{\epsilon}\rightarrow 0\,,
\end{equation*}
which establishes (\ref{S2.34}) and finishes the proof.
\end{proof}
\subsection{Sharp $L^p$-estimates for the operator $\mc{P}(g,h)$}
We are now in position to prove Theorem \ref{thm1} presented in the Introduction of this paper. Let $f$, $g$ and $h$ be in $C_0(\R^n)$. From Lemma \ref{Thm2.2} and a standard approximation argument we see that inequality (\ref{S2.16}) is valid for general angular functions $b:[-1,1] \to \R^{+}$ satisfying the Grad's cut-off assumption (\ref{Intro13.1}).
\begin{proof}[Proof of Theorem \ref{thm1}] Let $r'$ be the dual exponent of $r$. By duality and H\"{o}lder's inequality, together with Lemma \ref{Thm2.2} applied to a function $f_1(k) = f(k)|k|^{\alpha}$, where $f \in C_0(\R^n)$ and vanishes in a neighborhood of the origin, we have
\begin{align}\label{S2.45}
\begin{split}
\|\mathcal{P}&(g,h)\|_{L^r(\R^n, \dnua)}  =  \displaystyle\sup_{\|f\|_{r', \dnua}= 1} \left|\int_{\mathbb{R}^{n}}f(k)\,\mathcal{P}(g,h)(k)\,\dnua(k) \right|\\
& \leq \displaystyle\sup_{\|f\|_{r',\dnua} = 1}\int_{\mathbb{R}^{n}}\fs_{r'}(k)\,\mathcal{P}(\gs_p,\hs_q)(k)\, \dnua(k)  \  \leq \  \left\|\mathcal{P}(\gs_p,\hs_q)\right\|_{L^r(\R^n, \dnua)}.
\end{split}
\end{align}
Combining (\ref{S2.45}) with (\ref{S2.24}), (\ref{S2.27}), (\ref{S2.30}) we obtain
\begin{align*}
\begin{split}
&\left\|\mathcal{P}(\gs_p,\hs_q)\right\|_{L^r(\R^n, \dnua)} = \Bigl(\tfrac{\left|S^{n-1}\right|}{2}\Bigr)^{\tfrac{1}{r}}\,\left\|\widetilde{\mc{P}(\gs_p,\hs_q)}\right\|_{L^{r}(\R^{+},\,\dsig)}\\
& = \Bigl(\tfrac{\left|S^{n-1}\right|}{2}\Bigr)^{\tfrac{1}{r}}2^{n-2}\left|S^{n-2}\right| \left\|\mc{B} (\widetilde{\gs_p}, \widetilde{\hs_q})\right\|_{L^{r}(\R^{+},\,\dsig)}\\
& \leq \Bigl(\tfrac{\left|S^{n-1}\right|}{2}\Bigr)^{\tfrac{1}{r}}2^{n-2}\left|S^{n-2}\right| \beta_{b}\left(-\tfrac{n+\alpha}{2p},-\tfrac{n+\alpha}{2q}\right) \left\|\widetilde{\gs_p}\right\|_{L^p(\R^{+},\, \dsig)}\left\|\widetilde{\hs_q}\right\|_{L^q(\R^{+},\, \dsig)}\\
&= 2^{n-2}\left|S^{n-2}\right| \beta_{b}\left(-\tfrac{n+\alpha}{2p},-\tfrac{n+\alpha}{2q}\right)  \left\|\gs_p\right\|_{L^p(\R^n,\, \dnua)}\left\|\hs_q\right\|_{L^q(\R^n,\, \dnua)}\\
& = 2^{n-2}\left|S^{n-2}\right| \beta_{b}\left(-\tfrac{n+\alpha}{2p},-\tfrac{n+\alpha}{2q}\right)  \left\|g\right\|_{L^p(\R^n,\, \dnua)}\left\|h\right\|_{L^q(\R^n,\, \dnua)}.
\end{split}
\end{align*}
The fact that the constant is sharp follows easily from the sequence of functions constructed in the proof of Lemma \ref{Thm2.3}.
\end{proof}


\section{Young's inequality for the gain collision operator}

The goal of this section is to prove the Young's inequality for the gain term of the Boltzmann collision operator $Q^{+}$, in the case of hard potentials, proposed in Theorem \ref{thm2}. We start with a simple lemma that relates the full collision operator and the Maxwellian molecules operator by means of the operator $\mc{P}(g,h)$ studied in Section 2 (this is related to equation (2.6) in \cite{MV}). Throughout this section we may assume that all the functions are nonnegative (motivated by the solutions of the Boltzmann equation) to avoid technicalities when defining some integrals. 

In what follows we denote the translation and reflection by
$$\tau_{v}g(x):= g(x-v) \ \ \ \ \textrm{and} \ \ \ \ \ \mc{R}g(x):= g(-x).$$
\begin{lemma}\label{lem3.1}
Assume that the kernel 
\begin{equation*}\label{S3.1}
B(|u|,\hat{u}\cdot \omega)=\Phi(|u|)b(\hat{u}\cdot \omega)
\end{equation*}
satisfies $\Phi \in C(\R^{+})$ and $b \in C([-1,1])$. Assume also that $g,h \in C_0(\R^n)$. The full collision and the Maxwellian molecules operator are related by the formula
\begin{equation}\label{S3.2}
Q^{+}(g,h)(v)=\int_{\mathbb{R}^{n}}\Phi(|k|)\,\mathcal{P}\left(\tau_{-v}\,g,\tau_{-v}\,h\right)(k)\,\dk.
\end{equation}
\end{lemma}

\begin{proof}
This relation is a consequence of Carleman's representation (\cite[Lemma 16]{GPV})
\begin{align}\label{Car}
\begin{split}
Q^{+}&(g,h)(v)=\\
& 2^{n-1}\int_{\mathbb{R}^{n}}\frac{g(x+v)}{\left|x\right|}\int_{\left\{x\cdot z=0\right\}}\frac{h(z+v)}{\left|x+z\right|^{n-2}}\,B\left(-(x+z),\tfrac{x+z}{|x+z|}\cdot\tfrac{x-z}{|x+z|}\right) \dpi_z\,\dx.
\end{split}
\end{align}
Note that if $B(|u|,\hat{u}\cdot \omega)$ can be expressed as a product of a magnitude part by an angular part one obtains
\begin{equation*}
B\left(-(x+z),\tfrac{x+z}{|x+z|}\cdot\tfrac{x-z}{|x+z|}\right)=\Phi(|x+z|)\,b\left(\tfrac{x+z}{|x+z|}\cdot\tfrac{x-z}{|x+z|}\right).
\end{equation*}
However, in the hyperplane $\{x\cdot z=0\}$ the angular part reduces to
\begin{equation*}
\frac{x+z}{|x+z|}\cdot\frac{x-z}{|x+z|}=\frac{2|x|^{2}}{|x+z|^{2}}-1\,,
\end{equation*}
and we conclude that
\begin{align}\label{S3.3}
\begin{split}
&Q^{+}(g,h)(v)=\\
& 2^{n-1}\int_{\mathbb{R}^{n}}\frac{g(x+v)}{\left|x\right|}\int_{\left\{x\cdot z=0\right\}}\frac{\Phi(|x+z|)}{\left|x+z\right|^{n-2}}\;h(z+v)\;b\left(\tfrac{2|x|^{2}}{|x+z|^{2}}-1\right)\dpi_z\,\dx.
\end{split}
\end{align}
Now it is just a matter of comparing the expressions (\ref{S3.3}) and (\ref{S2.6}).
\end{proof}

\begin{proof}[Proof of Theorem \ref{thm2}]
First we consider $f,g,h \in C_0(\R^n)$ and $b \in C([-1,1])$. From Lemma \ref{lem3.1} we can write
\begin{equation*}
I:= \int_{\mathbb{R}^{n}}f(v)Q^{+}(g,h)(v)\,\dv=\int_{\mathbb{R}^{n}}\int_{\mathbb{R}^{n}}f(v)\mathcal{P}(\tau_{-v}g,\tau_{-v}h)\,|k|^{\lambda} \dk\,\dv\, ,
\end{equation*}
and from the definition of the operator $\mc{P}$ in (\ref{S2.3}), with a change of variables $v \to v-k^{+}$ we obtain
\begin{equation*}
I = \int_{\mathbb{R}^{n}}\int_{\mathbb{R}^{n}}\int_{S^{n-1}}g(v)f(v-k^{+})h(v-|k|\omega)b(\hat{k}\cdot\omega)\,\dom\, |k|^{\lambda}\dk\,\dv.
\end{equation*}
We now transform the integration on $k$ into polar coordinates
\begin{equation*}
I = \int_{\mathbb{R}^{n}}\int_{0}^{\infty} \int_{S^{n-1}}\int_{S^{n-1}}g(v)f(v-k^{+})h(v-|k|\omega)b(\hat{k}\cdot\omega)\,\dom\, \textrm{d}\hat{k}\, |k|^{\lambda+n-1}\textrm{d}|k|\,\dv.
\end{equation*}
By defining $x = |k|\omega$ we come back from polar coordinates to 
\begin{equation*}
I = \int_{\mathbb{R}^{n}}\int_{\R^n}\int_{S^{n-1}}g(v)f(v-x^{+})h(v-x)b(\hat{k}\cdot\hat{x})\, \textrm{d}\hat{k}\, |x|^{\lambda}\dx\,\dv\,,
\end{equation*}
and finally, by just relabeling the variables we arrive at the form that will be convenient to us
\begin{align*}
\begin{split}
I = &\int_{\mathbb{R}^{n}}\int_{\R^n}g(v)h(v-x)\left(\int_{S^{n-1}}f(v-x^{+})b(\hat{k}\cdot\hat{x})\, \textrm{d}\hat{k}\right)\, |x|^{\lambda}\dx\,\dv\,\\
= & \int_{\mathbb{R}^{n}}\int_{\R^n}g(v)h(v-k)\, \mc{P}(\mc{R}\tau_{-v}f,1)(k)\, |k|^{\lambda}\dk\,\dv\,.
\end{split}
\end{align*}
Using the inequality
$$|k|^{\lambda} \leq 2^{\lambda}\left(|v|^{\lambda} + |v-k|^{\lambda}\right),$$
for $\lambda\geq0$, we conclude that 
\begin{align}\label{S3.7}
 \begin{split}
I  &\ \leq \ \  2^{\lambda} \int_{\mathbb{R}^{n}}\int_{\R^n}g(v)|v|^{\lambda} h(v-k)\, \mc{P}(\mc{R}\tau_{-v}f,1)(k)\,\dk\,\dv \\
&  \ \ \ \ \ \ \ \ +  2^{\lambda} \int_{\mathbb{R}^{n}}\int_{\R^n}g(v) h(v-k)|v-k|^{\lambda}\, \mc{P}(\mc{R}\tau_{-v}f,1)(k)\,\dk\,\dv\\
&:= A + B\,.
 \end{split}
\end{align}
Our objective now is to bound conveniently the expressions $A$ and $B$ appearing in (\ref{S3.7}). This will be accomplished by means of H\"{o}lder's inequality with exponents $1/p' + 1/q' + 1/r = 1$ and Theorem \ref{thm1}. We simplify the notation by writing $g_{\lambda}(v) = g(v)|v|^{\lambda} $, and start with the analysis of $A$,
\begin{align}\label{S3.7.1}
\begin{split}
A & = 2^{\lambda} \int_{\mathbb{R}^{n}}\int_{\R^n}\left(g_\lambda(v)^{\tfrac{p}{r}} h(v-k)^{\tfrac{q}{r}}\right)\,\left(g_\lambda(v)^{\tfrac{p}{q'}} \mc{P}(\mc{R}\tau_{-v}f,1)(k)^{\tfrac{r'}{q'}}\right) \\
& \ \ \ \ \ \ \ \ \ \ \ \  \ \ \ \ \ \ \ \ \ \ \ \ \ \ \  \ \ \ \ \ \ \ \ \  \left(h(v-k)^{\tfrac{q}{p'}}\mc{P}(\mc{R}\tau_{-v}f,1)(k)^{\tfrac{r'}{p'}}\right)\,\dk\,\dv\\
& \leq 2^{\lambda}\left( \int_{\R^n} \int_{\R^n} g_\lambda(v)^{p} h(v-k)^{q} \dk\,\dv\right)^{\tfrac{1}{r}}\left(\int_{\R^n} \int_{\R^n} g_\lambda(v)^{p} \mc{P}(\mc{R}\tau_{-v}f,1)(k)^{r'} \dk\,\dv\right)^{\tfrac{1}{q'}}\\
&  \ \ \ \ \ \ \ \ \ \ \ \ \ \ \ \ \ \ \ \ \ \ \ \ \ \ \ \ \ \ \ \ \ \ \ \ \ \left(\int_{\R^n}\int_{\R^n} h(v-k)^{q}\mc{P}(\mc{R}\tau_{-v}f,1)(k)^{r'}\dk\,\dv\right)^{\tfrac{1}{p'}}\\
&: = 2^{\lambda} \,A_1 \, \,A_2 \, \,A_3.
\end{split}
\end{align}
We now obtain the bounds for $A_i$, $i=1,2,3$. First observe that
\begin{equation}\label{S3.8}
A_1 = \left(\|h\|_{L^q(\R^n)}^q\|g_{\lambda}\|_{L^{p}(\R^n)}^p \right)^{\tfrac{1}{r}}.
\end{equation}
Using Theorem \ref{thm1} we find
\begin{align}\label{S3.9}
 \begin{split}
A_2 &= \left( \int_{\R^n} g_\lambda(v)^p \left\|\mc{P}(\mc{R}\tau_{-v}f,1)\right\|_{L^{r'}(\R^n)}^{r'} \dv \right)^{\tfrac{1}{q'}}\\
& \leq \left( \int_{\R^n} g_\lambda(v)^p\,  C_2^{r'}\, \|\mc{R}\tau_{-v}f\|_{L^{r'}(\R^n)}^{r'} \dv \right)^{\tfrac{1}{q'}} = \left(C_2^{r'}\, \|g_{\lambda}\|_{L^p(\R^n)}^p \, \|f\|_{L^{r'}(\R^n)}^{r'} \right)^{\tfrac{1}{q'}}\,
 \end{split}
\end{align}
where the constant $C_2$ is given by Theorem \ref{thm1}
\begin{equation}\label{S3.10}
C_2 = 2^{n-2}\left|S^{n-2}\right|\beta_{b}\left(-\tfrac{n}{2r'},0\right) \leq 2^{n-2}\left|S^{n-2}\right|\beta_{b}\left(-\tfrac{n}{2r'},-\tfrac{n}{2r'}\right).
\end{equation}
The remaining term $A_3$, under the change of variables $v \to v+k$, becomes
\begin{align*}
\begin{split}
A_3  = \left(\int_{\R^n}\int_{\R^n} h(v)^{q}\mc{P}(1, \tau_{-v}f)(k)^{r'}\dk\,\dv\right)^{\tfrac{1}{p'}}\,,
\end{split}
\end{align*}
and therefore, using Theorem \ref{thm1} again,
\begin{align}\label{S3.11}
 \begin{split}
 A_3  & = \left(\int_{\R^n}h(v)^{q}\left\|\mc{P}(1, \tau_{-v}f)\right\|_{L^{r'}(\R^n)}^{r'}\dv\right)^{\tfrac{1}{p'}}\\
& \leq \left(\int_{\R^n}h(v)^{q}\,C_3^{r'}\,\|\tau_{-v}f\|_{L^{r'}(\R^n)}^{r'}\dv\right)^{\tfrac{1}{p'}} = \left(C_3^{r'}\, \|h\|_{L^q(\R^n)}^q \, \|f\|_{L^{r'}(\R^n)}^{r'} \right)^{\tfrac{1}{p'}},
 \end{split}
\end{align}
where the constant $C_3$ is given by
\begin{equation}\label{S3.12}
C_3 = 2^{n-2}\left|S^{n-2}\right|\beta_{b}\left(0, -\tfrac{n}{2r'}\right) \leq 2^{n-2}\left|S^{n-2}\right|\beta_{b}\left(-\tfrac{n}{2r'},-\tfrac{n}{2r'}\right).
\end{equation}
Combining expressions (\ref{S3.7.1})-(\ref{S3.12}) we obtain
\begin{align}\label{S3.13}
 \begin{split}
 A & \leq  2^{\lambda} A_1\,A_2\,A_3 \\
& \leq 2^{\lambda + n-2}\left|S^{n-2}\right|\beta_{b}\left(-\tfrac{n}{2r'},-\tfrac{n}{2r'}\right) \|f\|_{L^{r'}(\R^n)}\,\|g_{\lambda}\|_{L^p(\R^n)} \,\|h\|_{L^q(\R^n)}.
 \end{split}
\end{align}
Proceeding analogously for the $B$ term defined in (\ref{S3.7}) we will find
\begin{equation}\label{S3.14}
B \leq 2^{\lambda + n-2}\left|S^{n-2}\right|\beta_{b}\left(-\tfrac{n}{2r'},-\tfrac{n}{2r'}\right) \|f\|_{L^{r'}(\R^n)}\,\|g\|_{L^p(\R^n)} \,\|h_\lambda\|_{L^q(\R^n)}.
\end{equation}
Combining equations (\ref{S3.13}) and (\ref{S3.14}) we arrive at
\begin{align}\label{S3.15}
\begin{split}
I & \leq A + B \\
& \leq 2^{\lambda + n-2}\left|S^{n-2}\right|\beta_{b}\left(-\tfrac{n}{2r'},-\tfrac{n}{2r'}\right) \|f\|_{L^{r'}}\left(\|g_{\lambda}\|_{L^p} \,\|h\|_{L^q} + \|g\|_{L^p} \,\|h_\lambda\|_{L^q}\right)\\
& \leq 2^{\lambda + n-1}\left|S^{n-2}\right|\beta_{b}\left(-\tfrac{n}{2r'},-\tfrac{n}{2r'}\right)\|f\|_{L^{r'}(\R^n)}\|g\|_{{L_{\lambda}^p}(\R^n)}\|h\|_{{L_{\lambda}^q}(\R^n)}.
\end{split}
\end{align}
Inequality (\ref{Intro19}) now follows from (\ref{S3.15}) by duality. By a standard limiting argument (using monotone convergence, for example) we can extend inequality (\ref{Intro19}) for any angular kernel $b:[-1,1]\to \R^{+}$ that satisfies condition (\ref{Intro18}). This finishes the proof.

\end{proof}

\section*{Acknowledgments}
We would like to thank William Beckner for the fruitful discussions on the radial symmetrization techniques, especially on Lemma \ref{Thm2.2}, a core result in this paper. We are also thankful to Irene Gamba for her valuable suggestions and for bringing to our attention some of the references on the subject.

\end{document}